\documentclass{amsart}
\usepackage{amssymb}
\usepackage{amsfonts}
\usepackage{amssymb}
\usepackage{amsmath}
\usepackage{amsthm}
\usepackage{enumerate}
\usepackage{tabularx}
\usepackage{centernot}
\usepackage{mathtools}
\usepackage{stmaryrd}
\usepackage{amsthm,amssymb}
\usepackage{etoolbox}
\usepackage{url}
\usepackage{tikz}
\usepackage{amssymb}
\usetikzlibrary{matrix}
\usepackage{tikz-cd}
\usepackage{tikz}
\usepackage{marginnote}
\definecolor{mygray}{gray}{0.85}
\usepackage[backgroundcolor=mygray,colorinlistoftodos,prependcaption,textsize=small]{todonotes}
\usepackage{xargs}                      

\newcommand{\mrm}[1]{\mathrm{#1}}

\makeatletter
\def\subsection{\@startsection{subsection}{3}%
  \z@{.5\linespacing\@plus.7\linespacing}{.3\linespacing}%
  {\bfseries\centering}}
\makeatother

\makeatletter
\def\subsubsection{\@startsection{subsubsection}{3}%
  \z@{.5\linespacing\@plus.7\linespacing}{.3\linespacing}%
  {\centering}}
\makeatother

\makeatletter
\def\myfnt{\ifx\protect\@typeset@protect\expandafter\footnote\else\expandafter\@gobble\fi}
\makeatother

\newtheorem{theorem}{Theorem}[section]

\newtheorem{corollary}[theorem]{Corollary}
\newtheorem{definition}[theorem]{Definition}
\newtheorem{lemma}[theorem]{Lemma}

\newtheorem{question}[theorem]{Question}

\newtheorem{fact}[theorem]{Fact}

\newtheorem{remark}[theorem]{Remark}
\newtheorem{notation}[theorem]{Notation}
\newtheorem{conclusion}[theorem]{Conclusion}

\newtheorem{proviso}[theorem]{Proviso}
\newtheorem{conjecture}[theorem]{Conjecture}

\usepackage[backgroundcolor=mygray,colorinlistoftodos,prependcaption,textsize=small]{todonotes}
\usepackage{xcolor}

\newcounter{claimcounter}
\numberwithin{claimcounter}{theorem}

\newcommand{\pureindep}[1][]{%
  \mathrel{
    \mathop{
      \vcenter{
        \hbox{\oalign{\noalign{\kern-.3ex}\hfil$\vert$\hfil\cr
              \noalign{\kern-.7ex}
              $\smile$\cr\noalign{\kern-.3ex}}}
      }
    }\displaylimits_{#1}
  }
}

\begin{document}

\begin{abstract} 
We prove that if $A$ is a computable Hopfian finitely presented structure, then $A$ has a computable $d$-$\Sigma_2$ Scott sentence if and only if the weak Whitehead problem for $A$ is decidable. 
We use this to infer that every hyperbolic group as well as any polycyclic-by-finite group has a computable \mbox{$d$-$\Sigma_2$} Scott sentence, thus covering two main classes of finitely presented groups. Our proof also implies that every weakly Hopfian finitely presented group is strongly defined by its $\exists^+$-types, \mbox{a question which arose in a  different context.}
\end{abstract}

\title[{Computable Scott sentences and the Whitehead problem}]{Computable Scott sentences and the weak Whitehead problem for finitely presented groups}

\thanks{The author was partially supported by project PRIN 2022 ``Models, sets and classifications", prot. 2022TECZJA. We thank M. Bridson, J. Delgado, V. Guirardel, A. Martino and E. Ventura for useful discussions surrounding Conjecture~\ref{martin_theorem} and the algorithmic problems connected to it. We also thank D. Carolillo for useful comments \mbox{on an earlier version of this note}}

\author{Gianluca Paolini}
\address{Department of Mathematics ``Giuseppe Peano'', University of Torino, Italy.}
\email{gianluca.paolini@unito.it}


\date{\today}
\maketitle

\section{Introduction}

	In \cite{scott} Scott proved that every countable structure is characterized up to isomorphism by a sentence in $\mathfrak{L}_{\omega_1, \omega}$. Recently, there has been considerable interest \cite{alvir, knight1, trainor3, trainor2, trainor1, ho, knight2} in understanding how simple such a sentence can been chosen, in the sense of the following stratification of $\mathfrak{L}_{\omega_1, \omega}$ from {\em computable model theory}:
	\begin{enumerate}[(1)]
	\item $\varphi(\bar{x})$ is computable $\Pi_0$ and computable $\Sigma_0$ if it is finitary quantifier-free; 
	\item For an ordinal (resp. a computable ordinal) $\alpha >0$:
	\begin{enumerate}[(2.1)]
	\item $\varphi(\bar{x})$ is $\Sigma_\alpha$ (resp. computable $\Sigma_\alpha$) if it is a disjunction (resp. a computably enumerable disjunction) of formulas of the form $\exists\bar{y}\psi(\bar{x}, \bar{y})$, where $\psi$ is a $\Pi_\beta$-formula (resp. computable $\Pi_\beta$-formula) for some $\beta < \alpha$;
\end{enumerate}
	\begin{enumerate}[(2.2)]
	\item $\varphi(\bar{x})$ is $\Pi_\alpha$ (resp. computable $\Pi_\alpha$) if it is a conjunction (resp. a computably enumerable conjunction) of formulas of the form $\forall\bar{y}\psi(\bar{x}, \bar{y})$, where $\psi$ is a $\Sigma_\beta$-formula (resp. computable $\Sigma_\beta$-formula) for some $\beta < \alpha$;
\end{enumerate}
	\begin{enumerate}[(2.3)]
	\item $\varphi(\bar{x})$ is $d$-$\Sigma_\alpha$ if it is a conjunction of a $\Sigma_\alpha$-formula and a $\Pi_\alpha$-formula.
\end{enumerate}\end{enumerate}

	For a general overview of the motivation, history and literature surrounding the topic of optimal Scott sentences we refer the reader to the excellent survey \cite{trainor3}. In this paper we deal exclusively with optimal (computable) Scott sentences for finitely generated structures, and in particular for finitely presented groups; this topic has been central to the area. We shortly review the main progress in this direction. In \cite{knight2} it was proved that every finitely generated structure has a $\Sigma_3$ Scott sentence. In the same paper it was proved that several familiar finitely generated structures (e.g. free groups of finite rank, free abelian groups of finite rank, the infinite dihedral group) actually have a $d$-$\Sigma_2$ Scott sentence, and in fact a computable one. This led the authors to wonder if indeed every finitely generated group has a $d$-$\Sigma_2$ Scott sentence. This was answered negatively in \cite{trainor1} using methods from combinatorial group theory. In \cite{paolini} we isolated some general sufficient conditions for the existence of a computable $d$-$\Sigma_2$ Scott sentence, thus deducing e.g. that right-angled Coxeter groups of finite rank have a computable $d$-$\Sigma_2$ Scott sentence.
	Many questions on optimal Scott sentences for finitely generated structures are still open (see in particular \cite{trainor3, trainor2} for an overview). The following two questions appear to be the most \mbox{important in the area (cf. \cite[Questions~1 and 3]{trainor2} and references therein):} 
	
	\begin{question}\label{first_question} Is there a finitely presented group with no $d$-$\Sigma_2$ Scott sentence?
\end{question}
	
	\begin{question}\label{second_question} Is there a computable finitely presented group with a $d$-$\Sigma_2$ Scott sentence but no computable $d$-$\Sigma_2$ Scott sentence?
\end{question}

	For the rest of this introduction, we will shorten finitely presented as f.p. Concerning the first question, in \cite{trainor2} Harrison-Trainor gave a purely group theoretic characterization of the f.p. groups without a $d$-$\Sigma_2$ Scott sentence; despite this the question remains wide open at the moment. 
	Our paper deals instead with Question~\ref{second_question}. Concerning the existence of {\em computable} Scott sentences for f.p. (or more generally finitely generated) groups little appears to be known (cf. \cite[pg. 12]{trainor3}). An important conceptual result is in \cite{trainor2}, in this paper Harrison-Trainor constructs a finitely generated module with a $d$-$\Sigma_2$ Scott sentence but no computable $d$-$\Sigma_2$ Scott sentence. From this, using the universality among finitely generated structures of finitely generated groups \cite{trainor1}, Harrison-Trainor infers the abstract existence of a finitely generated group with a $d$-$\Sigma_2$ Scott sentence but no computable $d$-$\Sigma_2$ Scott sentence, but no explicit example of such a group can be exhibited using his argument. Furthermore, the argument of Harrison-Trainor from \cite{trainor2} says nothing on finite presentability. The only known general results that we are aware of come from our paper \cite{paolini}, already mentioned at the introduction, but the condition isolated there, although useful for some concrete applications, appeared to be of a rather technical nature. In this paper we show that, under the assumption of Hopfianity, the existence of a computable $d$-$\Sigma_2$ Scott sentence is equivalent to a weakening of a well-known algorithmic problem from combinatorial group theory: the {\em Whitehead problem}. This problem asks about the decidability of determining whether two tuples in a f.p. group are automorphic. The problem was first solved for free groups by Whitehead in \cite{whitehead_algo} (hence the name). There are now several important results of this kind, see in particular references \cite{orbit, the_hammer}. We define the {\em weak Whitehead problem}.
	
	\begin{definition}\label{def_Whitehead} Given a f.p. computable structure $A$ we say that $A$ has decidable weak Whitehead problem if for every finite generating tuple $\bar{a}$ of $A$ we have that the set $\{ \alpha(\bar{a}) : \alpha \in \mrm{Aut}(A) \}$ is a computable subset of $A^{n}$, where $n$ is the length of $\bar{a}$.
\end{definition}


	Our main result is the following characterization of the existence of a computable $d$-$\Sigma_2$ Scott sentence for a finitely presented structure (in a finite language).
	
	\begin{theorem}\label{main_theorem} 
	Let $A$ be a computable Hopfian f.p. structure. Then $A$ has a computable $d$-$\Sigma_2$ Scott sentence iff the weak Whitehead problem for $A$ is decidable.
\end{theorem}

	One of the virtues of Theorem~\ref{main_theorem} is that it allows us to use involved results in group theory to gain knowledge on the existence of computable $d$-$\Sigma_2$ Scott sentences. In particular, \ref{main_theorem} allows us to to settle the questions of existence of computable $d$-$\Sigma_2$ Scott sentences for two classes of finitely presented groups which have been at center of the attention among group theorists in the last decades, namely hyperbolic groups (see e.g. \cite{bridson}) and  polycyclic-by-finite groups (see e.g. \cite{segal}).

	\begin{corollary}\label{second_corollary} Hyperbolic groups have a computable $d$-$\Sigma_2$ Scott sentence.
\end{corollary}

	\begin{corollary}\label{third_corollary} Polycyclic-by-finite groups have a computable $d$-$\Sigma_2$ Scott sentence.
\end{corollary}

On the negative side, we propose the following conjecture:
	
		\begin{conjecture}\label{martin_theorem} There exists a Hopfian finitely presented group with decidable word problem (so a computable group) but undecidable weak Whitehead problem. 
\end{conjecture}
	
	 At first we thought that an affirmative answer to Conjecture~\ref{martin_theorem} was known but correspondence with experts and a careful reading of the literature on the subject convinced us otherwise. There are two algorithmic problems close to this problem: the above-mentioned Whitehead problem and the generation problem. The latter problem asks about the decidability of determining if a given tuple of elements from the group generates the whole group. It is known that in $F_2 \times F_2$ the Whitehead problem is decidable but the generation problem is not (cf. \cite[Theorem 4.4]{miller}). The difficulty about finding an example as in Conjecture~\ref{martin_theorem} is that both the decidability of the Whitehead problem and the decidability of the generation problem imply the decidability of the weak Whitehead problem, and so establishing undecidability of this latter problem is a harder task.
	

	
	
\medskip



	Our proof of \ref{main_theorem} amalgamates our proof from \cite{paolini} with some techniques from \cite{nies}. We end this introduction observing that our proof of Theorem~\ref{main_theorem} actually strengthens several results from \cite{types}, as well as resolving some of the open problems stated there, e.g. \cite[Problem~4]{types}. \mbox{In fact, in the terminology of \cite{types}, our proof implies:}
	
	\begin{theorem}\label{the_corrolary} Every weakly Hopfian f.p. group is strongly defined by its $\exists^+$-types. 
\end{theorem}

	For the notion of weakly Hopfianity we follow P.M. Neumann \cite{neumann}, for a definition see \ref{def_weak_Hopf}. 
	In Section~\ref{sec_types} we  explain the connection between our work and reference \cite{types}. After we finished writing this paper we discovered that Andr{\'e} already gave a solution to \cite[Problem~4]{types}; in fact in \cite{andre} he proves that if a finitely generated group $G$ is equationally noetherian, or finitely presented and Hopfian, then it is defined by its $\exists^+$-types. Our result \ref{the_corrolary} is much stronger than Andr{\'e}'s with respect to finitely presented groups, in fact, as mentioned in the first part of the introduction, it is a long-standing open problem whether there exists a f.p. non weakly Hopfian group.

\section{The proof of the characterization}\label{sec_proof}

	Our definition of finitely presented structure is standard; for details see e.g. \cite[Section~9.2]{hodges}. We assume that the underlying language $L$ is finite and contains no relational symbols. To express the presentation we write $A = \langle \bar{a} \mid \psi(\bar{x}) \rangle$, where the formula $\psi(\bar{x})$ is a conjunction of atomic positive $L$-formulas. The main example that we have in mind is of course the context of finitely presented groups (cf. \cite{presentation}).
	
	

	\begin{notation}\label{isolation_notation} In what follows we only consider finitely presented structures. Also, below we fix one such structure $A$ and one finite presentation $A = \langle \bar{a} \mid \psi(\bar{x}) \rangle$.
\end{notation}


	\begin{remark}\label{useful_remark} In the context of Notation~\ref{isolation_notation}, if $A \models \psi(\bar{b})$, then the map $\bar{a} \mapsto \bar{b}$ extends uniquely to an endomorphism of $A$.
\end{remark}

	\begin{definition}\label{T(g)_def} In the context of Notation~\ref{isolation_notation}, given $\bar{b} \in A^n$, we let:
	$$T(\bar{b}) = \{(t_1, ..., t_n) \in \mathrm{Term}^n: \exists \bar{c} \in A^n \text{ s.t. } A \models \psi(\bar{c}) \wedge \bigwedge_{i \in [1, n]} b_i = t_i(\bar{c})\};$$
	$$\widehat{T}(\bar{b}) = \mathrm{Term}^n \setminus T(\bar{b}),$$
where by $\mathrm{Term}$ we mean the set of $L$-terms (recall that the language $L$ is fixed).
\end{definition}

	\begin{lemma}\label{crucial_crucial_lemma}  Given $\bar{b}, \bar{c} \in A^n$ we have that the following are equivalent:
	\begin{enumerate}[(1)]
	\item $T(\bar{b}) \subseteq T(\bar{c})$;
	\item there is $\alpha \in \mathrm{End}(A)$ such that $\alpha(b_i) = c_i$, for all $i \in [1, n]$.
	\end{enumerate}
\end{lemma}

	\begin{proof} That (2) implies (1) is obvious. To prove the
other direction, recalling that $\bar{a} = (a_1,..., a_n)$ is fixed and determines a presentation of $A$, choose $(t_1,...,t_n) \in \mrm{Term}^n$ such that $b_i = t_i(a_1,..., a_n)$ for all $i \in [1, n]$. Since $T(\bar{b}     \subseteq T(\bar{c})$, we may pick $u_1,...,u_n \in A$ such that
$A \models \psi(u_1, ..., u_n) \wedge \bigwedge_{i \in [1, n]} c_i = t_i(u_1,..., u_n)$. Let now $\alpha$ be the endomorphism of $A$ given by $\alpha(a_i) = u_i$, then $\alpha(b_i) = c_i$, for all $i \in [1, n]$.
\end{proof}

\begin{definition}\label{def_weak_Hopf} Let $A$ be a structure.
	\begin{enumerate}[(1)]
	\item We say that $A$ is Hopfian if $\alpha \in \mrm{End}(A)$ and $\alpha$ is onto implies $\alpha \in \mrm{Aut}(A)$.
	\item We say that $A$ is weakly Hopfian if when $\alpha \in \mrm{End}(A)$ has a left inverse (i.e., there exists $\beta \in \mrm{End}(A)$ such that $\beta \circ \alpha = \mrm{id}_A$) we have that $\alpha \in \mrm{Aut}(A)$. 
\end{enumerate}	
\end{definition}

	\begin{lemma}\label{pre_crucial} Suppose that the structure $A$ is weakly Hopfian. Then, given $\bar{b} \in A^n$, the following are equivalent:
	\begin{enumerate}[(1)]
	\item $T(\bar{a}) = T(\bar{b})$;
	\item there is $\beta \in \mrm{Aut}(A)$ such that $\beta(\bar{a}) = \bar{b}$.
	\end{enumerate}
\end{lemma}

	\begin{proof} The fact that (2) implies (1) is obvious. Concerning the other direction, by \ref{crucial_crucial_lemma}, we can find $\alpha, \beta \in \mathrm{End}(A)$ such that $\beta(\bar{a}) = \bar{b}$ and $\alpha(\bar{b}) = \bar{a}$. Thus, necessarily $\alpha \circ \beta$ fixes the tuple $\bar{a}$, and so by the choice of $\bar{a}$ we have that $\alpha \circ \beta = \mrm{id}_A$. Hence, by weak Hopfianity of $A$ (cf. Definition~\ref{def_weak_Hopf}), $\alpha \in \mathrm{Aut}(A)$, and so $\beta \in \mathrm{Aut}(A)$.
\end{proof}

	\begin{lemma}\label{crucial_lemma} Suppose that the structure $A$ is weakly Hopfian and let $\Theta(\bar{x})$ be the following $\Pi_1$-formula:
	$$\psi(\bar{x}) \wedge \bigwedge_{\bar{t} \in \widehat{T}(\bar{a})} \forall \bar{y} \neg (\psi(\bar{y}) \wedge \bigwedge_{i \in [1, n]} x_i = t_{i}(\bar{y})).$$
Then, given $\bar{b} \in A^n$, the following are equivalent:
	\begin{enumerate}[(1)]
	\item $A \models \Theta(\bar{b})$;
	\item there is $\alpha \in \mrm{Aut}(A)$ such that $\alpha(\bar{a}) = \bar{b}$.
	\end{enumerate}
\end{lemma}

	\begin{proof} Clearly, $A \models \Theta(\bar{a})$, and so if $\alpha \in \mrm{Aut}(A)$ and $\bar{b} = \alpha(\bar{a})$, then $A \models \Theta(\bar{b})$. Concerning the other direction, let $\bar{b} \in A^n$ be such that $A \models \Theta(\bar{b})$. We claim that $T(\bar{b}) \subseteq T(\bar{a})$. For the sake of contradiction, suppose not, then there is $\bar{t} \in \widehat{T}(\bar{a})$ s.t.:
	$$A \models \exists \bar{y} (\psi(\bar{y}) \wedge \bigwedge_{i \in [1, n]} b_i = t_i(\bar{y})),$$ 
and so clearly it cannot be the case that $A \models \Theta(\bar{b})$. On the other hand, since $A \models \psi(\bar{b})$, by \ref{useful_remark} we have that the map $\bar{a} \mapsto \bar{b}$ extends to an endomorphism of $A$, and so by \ref{crucial_crucial_lemma} we have that $T(\bar{a}) \subseteq T(\bar{b})$. Putting everything together we have that $T(\bar{a}) = T(\bar{b})$. Hence, by Lemma~\ref{pre_crucial}, we are done.
\end{proof}


	\begin{lemma}\label{the_Hopfian_lemma} Suppose that $A$ is Hopfian. Then (recalling \ref{T(g)_def}) we have the following:
	$$T(\bar{a}) = \{(t_1, ..., t_n) \in \mathrm{Term}^n: a_i \mapsto t_i(\bar{a}) \in \mrm{Aut}(A) \}.$$
\end{lemma}

	\begin{proof} Suppose that $A$ is Hopfian. For the inclusion ``$\subseteq$'', let $\bar{t} = (t_1, ..., t_n) \in T(\bar{a})$, then there is $\bar{b} \in A^n$ such that $A \models \psi(\bar{b}) \wedge \bigwedge_{i \in [1, n]} a_i = t_i(\bar{b})$. Hence, the endomorphism $\alpha$ determined by $a_i \mapsto b_i$ is surjective, and thus, by Hopfianity, $\alpha \in \mrm{Aut}(A)$. Hence, we have $\alpha^{-1}(a_i) = t_i(\alpha^{-1}(\bar{b})) = t_i(\bar{a})$, and so we are done. For the inclusion ``$\supseteq$'', let $\bar{t} = (t_1, ..., t_n)$ be such that the endomorphism $\alpha$ determined by the assignment $a_i \mapsto t_i(\bar{a})$ is an automorphism. For $i \in [1, n]$, let $t_i(\bar{a}) = c_i$, then we have that $a_i = \alpha^{-1}(c_i) = t_i(\alpha^{-1}(\bar{a}))$, and so letting $\bar{b} = \alpha^{-1}(\bar{a})$ we have: 
	$$A \models \psi(\bar{b}) \wedge \bigwedge_{i \in [1, n]} a_i = t_i(\bar{b}).$$
\end{proof}

	\begin{lemma}\label{lemma_ce} Suppose that $A$ is computable and Hopfian. Then the following are equivalent:
	\begin{enumerate}[(1)]
	\item $\widehat{T}(\bar{a})$ (cf. \ref{T(g)_def}) is a computably enumerable subset of $\mrm{Term}^n$;
	\item the formula $\Theta(\bar{x})$ from Lemma~\ref{crucial_lemma} is a computable $\Pi_1$-formula;
	\item the set of tuples $\bar{b} \in A^{n}$ not in the $\mrm{Aut}(A)$-orbit of $\bar{a}$ is computably enumerable.	\end{enumerate}
\end{lemma}

	\begin{proof} Clearly (1) and (2) are equivalent. We show that (1) holds if and only if (3) does. Suppose that (3) holds, then, by \ref{the_Hopfian_lemma}, to check if $\bar{t} \in \widehat{T}(\bar{a})$ it suffices to let $b_i = t_i(\bar{a})$ and then check if $\bar{b} \in A^{n}$ is not in the $\mrm{Aut}(A)$-orbit of $\bar{a}$. Finally, suppose (1), then, again by \ref{the_Hopfian_lemma}, to check whether $\bar{b} \in A^{n}$ is not in the $\mrm{Aut}(A)$-orbit of $\bar{a}$ it suffices to search for $\bar{t} \in \widehat{T}(\bar{a})$ such that $A \models \bigwedge_{i \in [1, n]} b_i = t_i(\bar{a})$.	
%
\end{proof}

\begin{definition}\label{def_Whitehead} Given a finitely generated structure $A = \langle \bar{a} \rangle$ we say that $(A, \bar{a})$ has decidable weak Whitehead problem if there exists an algorithm which decides whether a tuple $\bar{b} \in G^n$ \mbox{is automorphic to $\bar{a}$ (i.e., if $\exists \alpha \in \mrm{Aut}(A)$ s.t. $\alpha(\bar{a}) = \bar{b}$).}
\end{definition}

\begin{lemma}\label{final_lemma} Suppose that $A$ is computable and Hopfian. Then the following are equivalent:
	\begin{enumerate}[(1)]
	\item the $\mrm{Aut}(A)$-orbit of $\bar{a}$ is definable by a computable $\Pi_1$-formula;
	\item the weak Whitehead problem for $(A, \bar{a})$ is decidable.
\end{enumerate}
\end{lemma}

	\begin{proof} The fact that (2) implies (1) is by Lemma~\ref{lemma_ce}. Concerning the other direction, let $\Delta(\bar{x})$ be a computable $\Pi_1$-formula defining the $\mrm{Aut}(A)$-orbit of $\bar{a}$. Evidently, deciding the weak Whitehead problem for $A = \langle \bar{a} \rangle$ means deciding (algorithmically) whether a tuple $\bar{b} \in A^n$ is or not in the following set:
	$$X_* = \{\bar{b} \in A^n : \exists \alpha \in \mrm{Aut}(A) \text{ such that } \alpha(\bar{a}) = \bar{b}\}.$$
We show that $X_*$ is computably enumerable and that $\widehat{X}_* := A^n \setminus X_*$ is computably enumerable, clearly this suffices (as then $X_*$ is computable). 
	\begin{enumerate}[$(\star_1)$]
	\item $\widehat{X}_*$ is computably enumerable.
	\end{enumerate}
Why $(\star_1)$? Recall that $\Delta(\bar{x})$ is a computable $\Pi_1$-formula defining the $\mrm{Aut}(A)$-orbit of $\bar{a}$ and so $\bar{b} \in X_*$ if and only if $A \models \Delta(\bar{b})$. It follows that $\bar{b} \in \widehat{X}_*$ iff $A \models \neg \Delta(\bar{b})$, but the negation of a $\Pi_1$-formula is a $\Sigma_1$-formula, and by \cite[Theorem~7.5]{hyper}, in a computable structure a $\Sigma_1$-formula always defines \mbox{a computably enumerable set.}
\begin{enumerate}[$(\star_2)$]
	\item ${X}_*$ is computably enumerable.
	\end{enumerate}
Why $(\star_2)$? Given $\bar{b} \in A^n$ search for $\bar{t} \in \mrm{Term}^n$ such that we have the following:
$$A \models \psi(\bar{b}) \wedge \bigwedge_{i \in [1, n]} a_i = t_i(\bar{b}).$$
As $A$ is assumed to be computable the procedure is algorithmic, and as $A$ is assumed to be Hopfian, by \ref{the_Hopfian_lemma}, the procedure stops exactly \mbox{when $\bar{b} \in {X}_*$, and so we are done.}
\end{proof}

	\begin{proviso} We now stop fixing the presentation $A = \langle \bar{a} \mid \varphi_1(\bar{a}), ..., \varphi_m(\bar{a}) \rangle$.
\end{proviso}

	\begin{fact}[{\cite[4.5]{alvir}}]\label{the_equivalence_fact} Let $A$ be a computable structure. Then TFAE:
	\begin{enumerate}[(1)]
	\item $A$ has a $d$-$\Sigma_2$ Scott sentence;
	\item $\exists$ $\bar{a}$ s.t. $\langle \bar{a} \rangle_A = A$ and the $\mrm{Aut}(A)$-orbit of $\bar{a}$ is \mbox{defined by a computable $\Pi_1$ fmla;}
	\item $\forall$ $\bar{a}$ s.t. $\langle \bar{a} \rangle_A = A$, the $\mrm{Aut}(A)$-orbit of $\bar{a}$ is defined by a computable $\Pi_1$ fmla.
	\end{enumerate}
\end{fact}

	\begin{fact}\label{finite_present_fact} Let $A$ be a finitely presented structure. Then for every finite tuple $\bar{a}$ generating $A$, $A$ can be presented by a \mbox{finite presentation in the generators $\bar{a}$.}
\end{fact}

	\begin{proof} This is a well-known general algebraic fact.
\end{proof}

	\begin{conclusion}\label{final_conclusion} Let $A$ be computable, Hopfian and finitely \mbox{presented. Then TFAE:}
	\begin{enumerate}[(1)]
	\item $\forall$ $\bar{a}$ s.t. $\langle \bar{a} \rangle_A = A$,  the weak Whitehead problem for $(A, \bar{a})$ is decidable;
	\item $\exists$ $\bar{a}$ s.t. $\langle \bar{a} \rangle_A = A$ such that the weak Whitehead problem for $(A, \bar{a})$ is decidable;
	\item $\exists$ $\bar{a}$ s.t. $\langle \bar{a} \rangle_A = A$ and the $\mrm{Aut}(A)$-orbit of $\bar{a}$ is \mbox{defined by a computable $\Pi_1$ frm;}
	\item $\forall$ $\bar{a}$ s.t. $\langle \bar{a} \rangle_A = A$, the $\mrm{Aut}(A)$-orbit of $\bar{a}$ is defined by a computable $\Pi_1$ frm.
\end{enumerate}
\end{conclusion}

	\begin{proof} This is clear from \ref{final_lemma}, \ref{finite_present_fact} and \ref{the_equivalence_fact}.
\end{proof}

	We recall the definition of $A$ having a decidable weak Whitehead problem.
	
	\begin{definition}\label{def_Whitehead+} Given a f.p. computable structure $A$ we say that $A$ has decidable weak Whitehead problem if for every finite generating tuple $\bar{a}$ of $A$ we have that the set $\{ \alpha(\bar{a}) : \alpha \in \mrm{Aut}(A) \}$ is a computable subset of $A^{n}$, where $n$ is the length of $\bar{a}$.
\end{definition}

	\begin{proof}[Proof of Theorem~\ref{main_theorem}] This is clear by \ref{the_equivalence_fact} and \ref{final_conclusion}.
\end{proof}

\begin{proof}[Proof of Corollary~\ref{second_corollary}] This follows from Theorem~\ref{main_theorem} and the solution of the isomorphism problem for hyperbolic groups \cite{the_hammer} (cf. also \cite[Theorem~2.7(ii)]{ventura_alone}).
\end{proof}

\begin{proof}[Proof of Corollary~\ref{third_corollary}] This follows from Theorem~\ref{main_theorem} and \cite[Theorem~A]{segal_decidability}).
\end{proof}

\section{Groups defined by their types}\label{sec_types}

	In \cite{types} the authors introduced the following notion (cf. \cite[Proposition~2]{types}):
	
	\begin{definition} Let $\mathcal{F}$ be a set of first-order formulas in the variables $\{x_i : i < \omega\}$ such that $\mathcal{F}$ contains all the positive quantifier-free formulas in the group theory language, $\mathcal{F}$ is closed under $\wedge$ and $\vee$, and $\mathcal{F}$ is closed under substitution of variables in words in $\{x_i : i < \omega\}$. We say that $f: A \rightarrow B$ is an $\mathcal{F}$-embedding if, for every $\varphi(\bar{a}) \in \mathcal{F}$ and $\bar{a} \in A^{\mrm{lg}(\bar{x})}$, we have that $A \models \varphi(\bar{x})$ if and only if $B \models \varphi(f(\bar{a}))$. We say that a finitely generated group $A$ (or, more generally, structure) is strongly defined by its $\mathcal{F}$-types iff \mbox{any $\mathcal{F}$-embedding $f: A \rightarrow A$ is an automorphism of $A$.}
	\end{definition} 
	
	\begin{notation} If $\mathcal{F}$ consists of the positive quantifier-free formulas, then an $\mathcal{F}$-embedding is referred to as an $\exists^+$-embedding (as usual in model theory literature).
\end{notation}

	\begin{lemma}\label{the_types_lemma} Let $A = \langle \bar{a} \mid \varphi_1(\bar{a}), ..., \varphi_m(\bar{a}) \rangle$ be as in Notation~\ref{isolation_notation} and let $f: A \rightarrow A$ be an $\exists^+$-embedding. Then $A \models \Theta(f(\bar{a}))$, where $\Theta(\bar{x})$ is as in Lemma~\ref{crucial_lemma}.
\end{lemma}

	\begin{proof} Let $f(\bar{a}) = \bar{b}$. Suppose not, then there is $\bar{t} \in \widehat{T}(\bar{a})$ such that we have:
	$$A \models \exists \bar{y} (\psi(\bar{y}) \wedge \bigwedge_{i \in [1, n]} b_i = t_i(\bar{y})).$$
But then, as $f: A \rightarrow A$ is an $\exists^+$-embedding and $f(\bar{a}) = \bar{b}$ we have that:
$$A \models \exists \bar{y} (\psi(\bar{y}) \wedge \bigwedge_{i \in [1, n]} a_i = t_i(\bar{y})),$$
but clearly this is a contradiction, as $\bar{t} \in \widehat{T}(\bar{a})$ (recalling Definition~\ref{T(g)_def}).
\end{proof}

	\begin{proof}[Proof of Theorem~\ref{the_corrolary}] This is immediate by \ref{crucial_lemma} and \ref{the_types_lemma}.
\end{proof}	

	Notice that the notion of being strongly defined  by its $\exists^+$-types is the strongest notion considered in \cite{types}, and so our result subsumes all the results of \cite{types}, if the group is assumed to be finitely presented. Furthermore, our result solves many of the questions they pose there (cf. \cite[pp. 5-6]{types}), in the case the group is finitely presented. In particular, the following specific corollary solves \cite[Problem~4]{types}.

	\begin{corollary} Every hyperbolic group is strongly defined  by its $\exists^+$-types.
\end{corollary}

	\begin{proof} It suffices to show that hyperbolic groups are finitely presented and Hopfian. The fact that they are finitely presented is well-known. Concerning being Hopfian, see e.g. \cite{sela}.
\end{proof}

\end{document}